\documentclass[a4paper]{amsart}
\usepackage{amsmath,amsthm,amssymb,latexsym,epic,bbm,comment,mathrsfs}
\usepackage{graphicx,enumerate,stmaryrd, color}
\usepackage[all,2cell]{xy}
\xyoption{2cell}

\newtheorem{theorem}{Theorem}
\newtheorem{lemma}[theorem]{Lemma}
\newtheorem{corollary}[theorem]{Corollary}
\newtheorem{proposition}[theorem]{Proposition}

\usepackage[all]{xy}
\usepackage[active]{srcltx}
\usepackage[parfill]{parskip}
\usepackage{enumerate}

\font\sc=rsfs10
\newcommand{\cC}{\sc\mbox{C}\hspace{1.0pt}}

\newcommand{\cI}{\sc\mbox{I}\hspace{1.0pt}}
\newcommand{\cJ}{\sc\mbox{J}\hspace{1.0pt}}
\newcommand{\cS}{\sc\mbox{S}\hspace{1.0pt}}

\font\scc=rsfs7

\newcommand{\ccS}{\scc\mbox{S}\hspace{1.0pt}}

\newcommand{\Hom}{\operatorname{Hom}}

\begin{document}
\title[$2$-representations of small quotients of 
Soergel bimodules]
{$2$-representations of small quotients\\ 
of Soergel bimodules in infinite types}

\author[H.~Ko and V.~Mazorchuk]
{Hankyung Ko and Volodymyr Mazorchuk}

\begin{abstract}
We determine for which Coxeter types the associated
small quotient of the $2$-category of Soergel bimodules is
finitary and, for such a small quotient, classify the simple 
transitive $2$-representations (sometimes under the additional
assumption of gradability). We also describe the underlying
categories of the simple transitive $2$-representations. 
For the small quotients of general Coxeter types, we give 
a description for the cell $2$-representations.
\end{abstract}

\maketitle

\section{Introduction and description of the results}\label{s1}

In this paper we fix the ground field $\mathbb{C}$ of complex numbers.
Let $(W,S)$ be a finitely generated Coxeter system and $\mathfrak{h}$ a reflections faithful
$W$-mo\-du\-le in the sense of \cite[Definition~1.5]{So2}. 
With this datum one associates a $2$-category
$\cS=\cS_{W,S,\mathfrak{h}}$ of Soergel bimodules, see \cite{So2}.
By \cite{So1} (for finite Weyl groups) and \cite{EW}
(in the general case), the Grothendieck decategorification of $\cS$
is isomorphic to the Hecke algebra $\mathbf{H}=\mathbf{H}_{W,S}$ of $W$.

The paper \cite{KMMZ} studies a certain quotient $\underline{\cS}$
of $\cS$, called {\em the small quotient}, in the case when the 
Coxeter system $(W,S)$ is finite. The main result of \cite{KMMZ}
is a classification of all simple transitive $2$-representations of
$\underline{\cS}$ in all finite Coxeter types but $I_2(12)$, $I_2(18)$ 
and $I_2(30)$. Under the additional assumption of gradability, 
classification of simple transitive $2$-representations
in these three exceptional cases was completed in \cite{MT}.

The setup of finite Coxeter systems, considered in \cite{KMMZ}, is
motivated by the fact that, in this setup, the $2$-category $\cS$
(and hence also the $2$-category $\underline{\cS}$) can be defined
over the coinvariant algebra of $W$ and is {\em finitary}
in the sense of \cite{MM1}. At the same time, for infinite Coxeter types,
it might still happen that the $2$-category $\underline{\cS}$ is finitary
despite the fact that the $2$-category $\cS$ no longer is.

In this note, we determine for which Coxeter systems the corresponding
small quotient of the $2$-category of Soergel bimodules is finitary, and,
in that case, classify all simple transitive $2$-representations
of this small quotient. Our argument is an application of the main result 
of \cite{MMMZ} which reduces our classification problem to the classification
results of \cite{KMMZ,MT}. Consequently, for some special cases, we need
to work under the additional assumption of gradability. We also
determine the quiver and relations for the underlying categories
of our $2$-representations. In the Coxeter type where the small quotient is not finitary, 
we do not have a classification for the simple transitive $2$-representations but still 
give the quiver description for a distinguished subclass of it, namely the cell $2$-representations.
 It turns out that, in all cases, the underlying
category corresponds to the zig-zag algebra (cf. \cite{HK,Du}) of a certain 
combinatorially described tree.

The paper is organized as follows: Section~\ref{s2} studies Kazhdan-Lusztig 
cell combinatorics of the small Kazhdan-Lusztig cell of an arbitrary
Coxeter system. Section~\ref{s3} is devoted to classification of
simple transitive $2$-representations of finitary small quotients of Soergel bimodules.
In Section~\ref{s4} one finds a description of the quiver and relations
for the underlying category of these $2$-representations.
\vspace{0.5cm}

{\bf Acknowledgements.}
The first author thanks Uppsala University for the hospitality during her visit where much of the work was done.
The second author is supported by the Swedish Research Council,
G{\"o}ran Gustafsson Foundation and Vergstiftelsen. We also thank Professor George Lusztig for
bringing the reference \cite{Lu} to our attention.

\section{Combinatorics of the small Kazhdan-Lusztig cell}\label{s2}

\subsection{Kazhdan-Lusztig cells}\label{s2.1}

We consider the {\em Hecke algebra} $\mathbf{H}=\mathbf{H}_{W,S}$ of $W$ in the normalization of
\cite{So2}. It has the {\em standard basis} $\{H_{w}:w\in W\}$ and the 
{\em Kazhdan-Lusztig (KL) basis} $\{\underline{H}_w:w\in W\}$, cf \cite{KL}.

We define the {\em left preorder} $\leq_L$ on $W$ as follows: for $x,y\in W$, we have $x\leq_L y$
provided that there is $w\in W$ such that $\underline{H}_y$ appears with a non-zero coefficient
when expanding $\underline{H}_w\underline{H}_x$ in the KL basis. Equivalence classes for
this preorder are called {\em left cells}. The fact that $x,y\in W$ belong to the same left 
cell is written $x\sim_L y$. Similarly one defines the {\em right preorder} $\leq_R$
and the right cells $\sim_R$ using $\underline{H}_x\underline{H}_w$, and also 
the {\em two-sided preorder} $\leq_J$ and the right cells $\sim_J$ using 
$\underline{H}_w\underline{H}_x\underline{H}_{w'}$.


\subsection{The small cell}\label{s2.2}

The following statement is fundamental for our results.
It can be easily obtained from \cite[Proposition~3.8]{Lu},
see also \cite{Do,Lu1,Lu2,Lu3}. A detailed
argument from \cite[Lemma~3 and Proposition~4]{KMMZ}, 
which is formally written in the setup of finite Coxeter groups, works in the general case.

\begin{proposition}\label{prop1}
Let $(W,S)$ be a finitely generated indecomposable Coxeter system.
 
\begin{enumerate}[$($i$)$]
\item\label{prop1.0} All simple reflections $s\in S$ belong to the same two-sided cell,
called the {\em small cell} and denoted $\mathcal{J}$.
\item\label{prop1.1} The map $\mathcal{L}\mapsto \mathcal{L}\cap S$ is a bijection between the set of
all left cells in $\mathcal{J}$ and $S$.
\item\label{prop1.3} The map $\mathcal{R}\mapsto \mathcal{R}\cap S$ is a bijection between the set of
all right cells in $\mathcal{J}$ and $S$.
\item\label{prop1.5} An element $e\neq w\in W$ belongs to $\mathcal{J}$ if and only if $w$
has a unique reduced expression. 
\item\label{prop1.7} If $\mathcal{L}$ is a left cell in  $\mathcal{J}$ with $s=\mathcal{L}\cap S$
and $\mathcal{R}$ is a right cell in  $\mathcal{J}$ with $t=\mathcal{R}\cap S$, then 
$\mathcal{L}\cap \mathcal{R}$ consists of all those element $w\in W$ with unique reduced expression
for which this unique expression is of the form $t\dots s$.
\end{enumerate}
\end{proposition}

For $s\in S$, we denote by $\mathcal{L}_s$ and $\mathcal{R}_s$ the left and the right cells 
containing $s$, respectively.

\subsection{Indecomposable Coxeter systems with finite small cells}\label{s2.3}

Let $(W,S)$ be a finitely generated indecomposable Coxeter system and $\Gamma=(\Gamma_0,\Gamma_1,m)$
is the associated Coxeter-Dynkin diagram. Here $\Gamma_0=S$ and $m:\Gamma_1\to\{3,4,\dots,\infty\}$
is the labeling function of the Coxeter presentation of $W$, that is, 
for any $\{s,t\}\in \Gamma_1$, we have  $(st)^{m_{s,t}}$ in the Coxeter presentation of $W$.
As usual, for $s,t\in S$, we have $\{s,t\}\not\in \Gamma_1$ if and only if $st=ts$.
We also omit the label $3$ and refer to an edge with label $3$ as {\em unlabeled}.

Indecomposability of $(W,S)$ is equivalent to the condition that $\Gamma$ is connected.
The following statement can be found in \cite[Proposition~3.8(h)]{Lu} without proof,
we include a proof for completeness.

\begin{proposition}\label{propfinite}
Let $(W,S)$ be a finitely generated indecomposable Coxeter system.
Then the small cell $\mathcal{J}$ is finite if and only if the following conditions are satisfied:
\begin{enumerate}[$($a$)$]
\item\label{propfinite.1} $\Gamma$ is a tree.
\item\label{propfinite.2} $\Gamma$ has at most one labeled vertex.
\item\label{propfinite.3} All labels of $\Gamma$ are finite.
\end{enumerate}
\end{proposition}

\begin{proof}
If $\Gamma$ contains a cycle, we can just take consecutive products of generators walking as 
long as we wish along this cycle and we never hit any of the relations in $W$. Therefore all elements of $W$ 
obtained in this way would belong to $\mathcal{J}$ by Proposition~\ref{prop1}\eqref{prop1.5},
making $\mathcal{J}$ infinite. Therefore condition~\eqref{propfinite.1} is necessary for finiteness of $\mathcal{J}$.

If $\Gamma$ contains a subgraph $\xymatrix{s\ar@{-}^{\infty}[r]&t}$, then all elements
of the form $stst\dots tsts$ contain no relations in $W$ and hence 
belong to $\mathcal{J}$ by Proposition~\ref{prop1}\eqref{prop1.5}.
This again makes $\mathcal{J}$ infinite and means that condition~\eqref{propfinite.3} is necessary
for finiteness of $\mathcal{J}$.

If $\Gamma$ contains a connected subgraph of the form
\begin{displaymath}
\xymatrix{
s_1\ar@{-}^{m}[r]&s_2\ar@{-}[r]&s_3\ar@{-}[r]&s_4\ar@{-}[r]&\dots\ar@{-}[r]&s_{k-1}\ar@{-}[r]^n&s_k
} 
\end{displaymath}
(with distinct vertices), where $k\geq 3$ and $m,n>3$, then all elements
of the form 
\begin{displaymath}
(s_1s_2s_3\dots s_{k-1}s_ks_{k-1}\dots s_3s_2)^l,\quad \text{ where }\quad l\in\{1,2,3,\dots\},
\end{displaymath}
contain no relations in $W$ and hence 
belong to $\mathcal{J}$ by Proposition~\ref{prop1}\eqref{prop1.5}.
This again makes $\mathcal{J}$ infinite and means that condition~\eqref{propfinite.2} is necessary
for finiteness of $\mathcal{J}$.

Let us now show that conditions~\eqref{propfinite.1}, \eqref{propfinite.2} and \eqref{propfinite.3}
together are sufficient for finiteness of $\mathcal{J}$. Let $w\in\mathcal{J}$ with the unique 
reduced expression $t_kt_{k-1}\dots t_1$ (here a repetition of simple reflections is allowed). 
To prove our claim we just need to establish some 
bound on $k$. As the reduced  expression of $w$ is unique, no pair of consecutive
elements in this expression can commute. This means that all pairs $\{t_2,t_{1}\}$, $\{t_3,t_{2}\}$, an so on,
are edges in $\Gamma$ and hence we can view $w$ as a walk in $\Gamma$ starting at $t_1$ in the
obvious way. As $\Gamma$ contains no cycles by \eqref{propfinite.1}, for big $k$ some edge $\{a,b\}$ has to be walked 
along in the way $aba$. To keep the reduced expression unique, the edge $\{a,b\}$ then must be
labeled. 

If $t_kt_{k-1}\dots t_1$ contains a subexpression of the form $abauaba$, where neither $a$ nor $b$
appears in $u$, then all edges appearing in $aua$ are unlabeled by \eqref{propfinite.2}.
Again, as $\Gamma$ contains no cycles by \eqref{propfinite.1}, if $u\neq e$, one of the edges in $u$
will have to be walked along in the way $a'b'a'$, contradicting uniqueness of the reduced
expression. Hence $u=e$ and this shows that all appearances of $a$ and $b$ in $w$ are next to each
other. 

The length of a subexpression of $w$ of the form $abab\dots$ is bounded by the finite label of
$\{a,b\}$, given by \eqref{propfinite.3}. By the argument above, the lengths of the subwords 
on both sides of this subexpression are bounded by the total number of edges. 
This implies that $k$ is bounded, and the claim follows.
\end{proof}

\subsection{Combinatorics of finite small cells}\label{s2.4}

In this subsection we assume that  $(W,S)$ is a finitely generated indecomposable Coxeter system
such that the small cell $\mathcal{J}$ is finite. Our aim here is to describe 
the intersections $\mathcal{L}_s\cap \mathcal{R}_t$, for $s,t\in S$.

\begin{corollary}\label{corcard1}
Let $(W,S)$ be a finitely generated indecomposable Coxeter system.
If $\Gamma$ is an unlabeled tree, then $|\mathcal{L}_s\cap \mathcal{R}_t|=1$, for all
$s,t\in S$.
\end{corollary}

\begin{proof}
We view $w\in \mathcal{L}_s\cap \mathcal{R}_t$ as a walk in $\Gamma$ similarly to the proof of
Proposition~\ref{propfinite}. As all edges in $\Gamma$ are unlabeled, the argument in the 
proof of Proposition~\ref{propfinite} implies that no edge can be repeated during this walk.
Therefore, by Proposition~\ref{prop1}\eqref{prop1.7},  
the only possibility for $w$ is to be the unique shortest path from $s$ 
to $t$ of the form $t\dots s$. The claim follows.
\end{proof}

Let us now assume that $\Gamma$ is a tree with a unique labeled edge $\{s,t\}$ of finite 
label $n\in\{4,5,\dots\}$. Let $\Gamma^{(s)}$ denote the connected component of $\Gamma\setminus\{t\}$
containing $s$ (or, equivalently, the full subgraph of $\Gamma$ consisting of
all vertices $r$ for which there is a walk from $r$ to $s$ which does not pass through $t$).
Let, similarly,  $\Gamma^{(t)}$ denote the connected component of $\Gamma\setminus\{s\}$
containing $t$.
Define the function $\pi:S\to\{s,t\}$ by sending vertices in $\Gamma^{(s)}$ to $s$
and vertices in $\Gamma^{(t)}$ to $t$.

\begin{proposition}\label{propcard2}
Assume that $\Gamma$ is a tree with a unique labeled edge $\{s,t\}$.
Then, for $p,q\in S$, there is a bijection between
$\mathcal{L}_{\pi(p)}\cap \mathcal{R}_{\pi(q)}$ and 
$\mathcal{L}_p\cap \mathcal{R}_q$
given by sending $w\in \mathcal{L}_{\pi(p)}\cap \mathcal{R}_{\pi(q)}$ to 
$u_1wu_2$, where $u_1$ is the unique shortest path from $\pi(q)$ to $q$
and $u_2$ is the unique shortest path from $p$ to $\pi(p)$.
\end{proposition}

\begin{proof}
Using  Proposition~\ref{prop1}\eqref{prop1.7}, it is easy to see that the map from 
the set $\mathcal{L}_{\pi(p)}\cap \mathcal{R}_{\pi(q)}$ to the set $\mathcal{L}_p\cap \mathcal{R}_q$
described in the formulation is well-defined. It is obviously injective. 
Furthermore, it is surjective due to the argument in the proof of 
Proposition~\ref{propfinite}. The claim follows.
\end{proof}

In the setup of Proposition~\eqref{propcard2}, let $(W^{\{s,t\}},\{s,t\})$ be the parabolic
Coxeter subsystem of $(W,S)$ corresponding to $\{s,t\}\subset S$. We will add the
superscript $\{s,t\}$ to objects associated with this Coxeter group, for example,
$\mathcal{L}_s^{\{s,t\}}$ means the left cell of $W^{\{s,t\}}$ containing $s$.

\begin{corollary}\label{corcard3}
Assume that $\Gamma$ is a tree with a unique labeled edge $\{s,t\}$.
Then, for $p,q\in \{s,t\}$, we have
\begin{displaymath}
\mathcal{L}_p\cap \mathcal{R}_q =\mathcal{L}_p^{\{s,t\}}\cap \mathcal{R}_q^{\{s,t\}}.
\end{displaymath}
\end{corollary}

\begin{proof}
That every element in $\mathcal{L}_p^{\{s,t\}}\cap \mathcal{R}_q^{\{s,t\}}$
belongs to $\mathcal{L}_p\cap \mathcal{R}_q$ is clear from the definitions.
That  every element in  $\mathcal{L}_p\cap \mathcal{R}_q$ belongs to
$\mathcal{L}_p^{\{s,t\}}\cap \mathcal{R}_q^{\{s,t\}}$ follows from
the argument in the proof of  Proposition~\ref{propfinite}.
\end{proof}

\subsection{Examples}\label{s2.5}

If $\Gamma$ is given by 
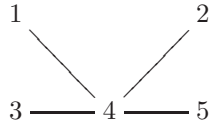
\begin{displaymath}
\xymatrix{ 
1\ar@{-}[rd]&&2\ar@{-}[ld]\\
3\ar@{-}[r]&4&5\ar@{-}[l],
}
\end{displaymath}
then the cell $\mathcal{J}$ has the following structure (here columns are left cells and 
rows are right cells indexed by the corresponding simple reflections):
\begin{displaymath}
\begin{array}{c||c|c|c|c|c}
&1&2&3&4&5\\ 
\hline\hline
1&1&142&143&14&145\\ 
\hline
2&241&2&243&24&245\\ 
\hline
3&341&342&3&34&345\\ 
\hline
4&41&42&43&4&45\\ 
\hline
5&541&542&543&54&5
\end{array}
\end{displaymath}

If $\Gamma$ is given by 
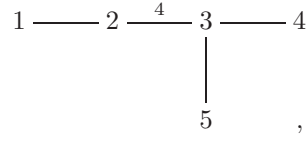
\begin{displaymath}
\xymatrix{ 
1\ar@{-}[r]&2\ar@{-}[r]^4&3&4\ar@{-}[l]\\
&&5\ar@{-}[u]&,
}
\end{displaymath}
then the cell $\mathcal{J}$ has the following structure (here columns are left cells and 
rows are right cells indexed by the corresponding simple reflections):
\begin{displaymath}
\begin{array}{c||c|c|c|c|c}
&1&2&3&4&5\\ 
\hline\hline
1&1,12321&12,1232&123&1234&1235\\ 
\hline
2&21,212321&2,232&23&234&235\\ 
\hline
3&321&32&3,323&34,3234&35,3235\\ 
\hline
4&4321&432&43,4323&4,43234&435,43235\\ 
\hline
5&5321&532&53,5323&534,53234&5,53235
\end{array}
\end{displaymath}

\section{Small quotients of Soergel bimodules and their $2$-representations}\label{s3}

\subsection{Small quotients of Soergel bimodules}\label{s3.1}
Below, by  {\em graded} we always mean {\em $\mathbb{Z}$-graded}. We also work over $\mathbb{C}$.

Let $\cS^{\mathrm{gr}}:=\cS^{\mathrm{gr}}(W,S,\mathfrak{h})$ 
denote the monoidal category of (graded) Soergel bimodules over the polynomial
algebra  $\mathbb{C}[\mathfrak{h}^*]$ associated to $(W,S)$ and $\mathfrak{h}$, see \cite{So2,EW}.
By \cite{So2,EW}, the split Grothendieck ring of $\cS^{\mathrm{gr}}$ is isomorphic to $\mathbf{H}$
and, for $w\in W$, we denote by $B_w$ the unique (up to isomorphism) indecomposable Soergel bimodule which corresponds
to the element $\underline{H}_w$ under this isomorphism. Our normalization of grading shifts
is chosen such that the full subcategory $\mathcal{P}$
of $\cS^{\mathrm{gr}}$ with objects $\{B_w:w\in W\}$ is positively graded. The graded category $\cS^{\mathrm{gr}}$ has finite dimensional graded components, that is, the morphism space between any two indecomposable 1-morphisms is finite dimensional.

\begin{lemma}\label{lem1}
The monoidal category $\cS^{\mathrm{gr}}$ has a unique tensor ideal $\cI$ which is maximal 
in the set of all graded tensor ideals of $\cS^{\mathrm{gr}}$ having the property that they 
\begin{equation}\label{eq1}
\text{do not contain any }\mathrm{id}_{B_w}, \text{ where }w\in \mathcal{J}.
\end{equation}
\end{lemma}

\begin{proof}
Due to positivity of the grading on $\mathcal{P}$, the sum of any family of ideals having property~\eqref{eq1}
also has property~\eqref{eq1}. Therefore $\cI$ is just the sum of all tensor ideals 
having property~\eqref{eq1}.
\end{proof}

The quotient monoidal category $\underline{\cS}^{\mathrm{gr}}:=\cS^{\mathrm{gr}}/\cI$
is called the {\em small quotient} of $\cS^{\mathrm{gr}}$, cf. \cite{KMMZ}. We denote by
$\cS$ and $\underline{\cS}$ the ungraded versions of $\cS^{\mathrm{gr}}$ and 
$\underline{\cS}^{\mathrm{gr}}$, respectively.

By construction, the images of $B_w$, where $w\in\mathcal{J}\cup\{e\}$, form a
complete and irredundant list of representatives of
the isomorphism classes of indecomposable objects in $\underline{\cS}$.

\subsection{$2$-categories and $2$-representations}\label{s3.2}

Although it is natural to define $\cS^{\mathrm{gr}}$ as a tensor category, we would like to 
adapt to the setup of $2$-representations of $2$-categories considered in \cite{MM1,MM3}.
For this we use the coherence theorem for monoidal categories and consider a strictification 
of $\cS^{\mathrm{gr}}$ which we will denote by the same symbol, abusing notation. 
We also identify a strict monoidal category with the corresponding $2$-category  with one object.

Recall, from \cite{MM1}, that a finitary $2$-category (over $\mathbb{C}$) 
is a $2$-category $\cC$ such that
\begin{itemize}
\item $\cC$ has finitely many objects;
\item each $\cC(\mathtt{i},\mathtt{j})$ is equivalent to the category of projective modules over
some finite dimensional $\mathbb{C}$-algebra;
\item all compositions are biadditive and $\mathbb{C}$-bilinear whenever this makes sense and
all identity $1$-morphisms are indecomposable.
\end{itemize}
The $2$-category $\cS$ is never finitary as homomorphism spaces between Soergel
bimodules are infinite dimensional vector spaces. However, the small quotient $\underline{\cS}$ is finitary
in some cases, as we will see in Subsection~\ref{s3.4} below.

A $2$-representation of a $2$-category $\cC$ is a $2$-functor from $\cC$
to an appropriate $2$-category, see \cite{MM3}. This can also be viewed as a functorial action
of $\cC$ on suitable categories. All $2$-representation of $\cC$
form a $2$-category in a natural way, see \cite{MM3} for details.
For every object $\mathtt{i}\in\cC$, we have the corresponding {\em principal}
$2$-representation $\cC(\mathtt{i},{}_-)$.

\subsection{A version of graded ``abelianization''}\label{s3.3}

There is a natural diagrammatic abelianization $2$-functor for finitary $2$-categories,
see \cite[Subsection~4.2]{MM2}.  Here we will need a slight modification of this 
functor due to the fact that $\cS^{\mathrm{gr}}$ is not finitary. 
We only need to abelianize 
the principal $2$-representation of $\cS^{\mathrm{gr}}$, so we will try to present the
object we need with a minimum amount of technicalities. For $i\in\mathbb{Z}$, we denote by 
$\langle i\rangle$ the corresponding functor of grading shift (which maps elements of degree 
$n$ to elements of degree $n-i$, for all $n\in\mathbb{N}$). Here we view $\cS^{\mathrm{gr}}$ as a monoidal category rather than a 2-category and define its abelianization as a 1-category.

We denote by $\overline{\cS^{\mathrm{gr}}}$ the category whose endomorphisms are diagrams of the form
\begin{displaymath}
\big(\prod_{i\leq i_0}\mathrm{F}_i\langle i\rangle\big)\to \mathrm{G}, \text{ where } i_0\in\mathbb{Z},
\end{displaymath}
in $\cS^{\mathrm{gr}}$,
and whose morphisms are given by the commutative solid squares in $\cS^{\mathrm{gr}}$
\begin{displaymath}
\xymatrix{
\displaystyle
\big(\prod_{i\leq i_0}\mathrm{F}_i\langle i\rangle\big)\ar[rr]\ar[d]&&\mathrm{G}\ar[d]^{\alpha}\ar@{.>}[dll]_{\beta}\\
\displaystyle
\big(\prod_{j\leq i'	_0}\mathrm{F}'_j\langle j\rangle\big)\ar[rr]&&\mathrm{G}'
}
\end{displaymath}
modulo those in which $\alpha$ factors through some $\beta$. It is important to note that,  due to our positive grading and the bound on the grading shift, for each $i$, the $1$-morphism $\mathrm{F}_i\langle i\rangle$ 
(as well as the $1$-morphism $\mathrm{G}$) has a non-zero 
homomorphism only to finitely many $1$-morphisms $\mathrm{F}'_j\langle j\rangle$.

By construction, the category $\overline{\cS^{\mathrm{gr}}}$ is equivalent to the category of 
finitely generated graded $\mathcal{P}^{\mathrm{op}}$-modules. 
In particular, $\overline{\cS^{\mathrm{gr}}}$ is abelian
whenever $\mathcal{P}^{\mathrm{op}}$ is noetherian. 
We do not know when $\mathcal{P}^{\mathrm{op}}$ is noetherian, but we do not need $\overline{\cS^{\mathrm{gr}}}$ to be abelian in what follows.

The left regular action of $\cS^{\mathrm{gr}}$ on itself extends, in the obvious way, to an
action of $\cS^{\mathrm{gr}}$ on $\overline{\cS^{\mathrm{gr}}}$.

\subsection{Finitary small quotients of Soergel bimodules}\label{s3.4}

\begin{proposition}\label{prop2}
Assume that $(W,S)$ is indecomposable. Then the $2$-category $\underline{\cS}$ is finitary
if and only if $\mathcal{J}$ is finite.
\end{proposition}

\begin{proof}
If  $\underline{\cS}$ is finitary, it must contain only finitely many isomorphism classes of
indecomposable objects. By construction, the latter are indexed by $\mathcal{J}\cup\{e\}$.
Therefore the condition $|\mathcal{J}|<\infty$ is necessary.

To prove that the condition $|\mathcal{J}|<\infty$ is sufficient, we assume that this condition is satisfied 
and we need to show that all homomorphism spaces in  $\underline{\cS}$ are finite dimensional.

Let $\overline{\cS^{\mathrm{gr}}}$ be the diagrammatic abelianization of $\cS^{\mathrm{gr}}$
as in Subsection~\ref{s3.3}. For $w\in W$, we denote by $L_w$ the simple top of 
the projective object $B_w$ in $\overline{\cS^{\mathrm{gr}}}$.
Fix some $s\in S$ and consider the full subcategory $\mathcal{A}$ of $\overline{\cS^{\mathrm{gr}}}$
given by the additive  closure of the objects of the form $B_w L_s$, where $w\in\mathcal{L}_s$ 
(up to graded shift). Similarly to \cite[Lemma~12]{MM1}, one shows that, for $w\in W$,
the inequality $B_w L_s\neq 0$ implies $w\in \mathcal{L}_s\cup\{e\}$. Consequently, the action of
$\cS^{\mathrm{gr}}$ on $\overline{\cS^{\mathrm{gr}}}$ restricts to $\mathcal{A}$.

Our next observation is that each $B_x L_s$, where $x\in \mathcal{L}_s$, is finite dimensional in the sense that the direct sum
\[\bigoplus_{y\in W,i\in \mathbb Z}\mathrm{Hom}_{\overline{\ccS^{\mathrm{gr}}}}(B_y\langle i\rangle,B_x L_s)\]
is finite dimensional. Indeed, for $y\in W$ and $i\in\mathbb{Z}$, using adjunction we have 
\begin{displaymath}
\mathrm{Hom}_{\overline{\ccS^{\mathrm{gr}}}}(B_y\langle i\rangle,B_x L_s)\cong
\mathrm{Hom}_{\overline{\ccS^{\mathrm{gr}}}}(B_{x^{-1}}B_y\langle i\rangle, L_s).
\end{displaymath}
if the right hand side is non-zero, then $y\in \mathcal{L}_s\cup\{e\}$ and $|i|$ 
must be bounded by the length of $x$. This leaves us with finitely many choices for both
$x$ and $i$. 
Furthermore, all graded components of homomorphism spaces in 
$\overline{\cS^{\mathrm{gr}}}$ are finite dimensional.

As each $B_x L_s$, where $x\in \mathcal{L}_s$, is finite dimensional, it has
finitely many indecomposable summands. Therefore, up to grading shift, $\mathcal{A}$ has 
finitely many indecomposable objects. In other words, the action of $\cS^{\mathrm{gr}}$ 
on $\mathcal{A}$ is an action on some category which is equivalent to the category of 
graded projective modules over a finite dimensional graded algebra. Let $\cJ$ be the
kernel (in $\cS^{\mathrm{gr}}$) of this action. Note also that this action is given by exact 
functors as all $1$-morphisms in $\cS^{\mathrm{gr}}$ have biadjoints.
This implies that all morphism spaces between $1$-morphisms in
the ungraded version of $\cS^{\mathrm{gr}}/\cJ$ are finite dimensional.

Note that the ideal $\cJ$ is graded by construction and that the identity $2$-morphism on $B_s$
does not belong to $\cJ$ as $B_s(B_s L_s)\neq 0$. Hence $\cJ\subset \cI$ by the maximality
of $\cI$. Consequently, all morphism spaces between $1$-morphisms in
$\underline{\cS}$ are finite dimensional. This completes the proof.
\end{proof}

\subsection{Simple transitive $2$-representations of finitary small quotients of Soergel bimodules}\label{s3.5}

Following \cite{MM5}, we are interested in classification of {\em simple transitive $2$-representations}
of $\underline{\cS}$ in case the latter $2$-category is finitary. Recall, from \cite{MM5},
that a simple transitive $2$-representation of $\underline{\cS}$ is a functorial action 
of  $\underline{\cS}$ on a small category $\mathcal{C}$
equivalent to $B$-proj, for some finite dimensional algebra $B$, such that  $\mathcal{C}$ has no 
proper $\underline{\cS}$-invariant ideals.

Examples of simple transitive $2$-representations of $\underline{\cS}$ include the so-called
{\em cell $2$-rep\-re\-sen\-ta\-ti\-on} $\mathbf{C}_{\mathcal{L}}$ associated to a left cell 
$\mathcal{L}$, cf. \cite{MM1,MM2}.

\begin{proposition}\label{prop3}
Assume that $\Gamma$ is an unlabeled tree. Then every 
simple transitive $2$-representations of $\underline{\cS}$ is equivalent to a
cell $2$-representation. 
\end{proposition}

\begin{proof}
Thanks to Corollary~\ref{corcard1}, \cite[Proposition~1]{MM6} 
and \cite[Corollary~19]{KM}, in case  $\Gamma$ is an unlabeled tree, 
the $2$-category $\underline{\cS}$ satisfies the assumptions of \cite[Theorem~18]{MM5}
and hence the assertion of the proposition follows from \cite[Theorem~18]{MM5}.
\end{proof}

The case when $\Gamma$ is a tree with one labeled edge requires some preparation.
Assume that the labeled edge of gamma is
\begin{equation}\label{eq2}
\xymatrix{s\ar@{-}[r]^n&t}, 
\end{equation}
where $3<n<\infty$. Let $\widetilde{S}=\{s,t\}$ and $\widetilde{W}$ be the parabolic Coxeter subgroup of 
$W$ generated by $\widetilde{S}$. Let $\widetilde{\mathfrak{h}}$ be the $2$-dimensional subspace of $\mathfrak{h}$ 
generated by the unique (up to scalar) eigenvector of $s$ with eigenvalue $-1$ and
the unique (up to scalar) eigenvector of $t$ with eigenvalue $-1$. Then we have the corresponding
$2$-categories $\widetilde{\cS}^{\mathrm{gr}}$ and $\underline{\widetilde{\cS}}$. 

\begin{theorem}\label{thm5}
Assume that $\Gamma$ is a tree with one labeled edge of the form \eqref{eq2}.
Then there is a bijection between the sets of equivalence classes of simple transitive $2$-representations
of the $2$-categories $\underline{\cS}$ and $\underline{\widetilde{\cS}}$.
\end{theorem}

Note that, for $n\neq 12,18,30$, simple transitive $2$-representations of 
$\underline{\widetilde{\cS}}$ are classified in \cite[Theorem~1]{KMMZ}.
For $n=12,18,30$, simple transitive $2$-representations of 
$\underline{\widetilde{\cS}}$ are classified in \cite[Theorem~II]{MT} 
(with some classes of $2$-representations constructed in \cite[Theorem~1]{KMMZ}).

\begin{proof}
Our proof of this theorem is based crucially on an application of \cite[Theorem~15]{MMMZ}. 
Note that $\underline{\cS}$ is finitary by Propositions~\ref{propfinite} and \ref{prop2}.
Define $\cC$ as the quotient of the $2$-subcategory of $\underline{\cS}$ generated by 
$B_w$, where $w\in\mathcal{L}_s\cap\mathcal{R}_s$, modulo the unique maximal two-sided
$2$-ideal which does not contain any non-zero identity $2$-morphisms. By \cite[Theorem~15]{MMMZ},
there is a bijection between the sets of equivalence classes of simple transitive $2$-representations
of the $2$-categories $\underline{\cS}$ and $\cC$.

Define $\widetilde{\cC}$ as the quotient of the $2$-subcategory of $\underline{\widetilde{\cS}}$ generated by 
$B_w$, where $w\in\mathcal{L}_s\cap\mathcal{R}_s$, modulo the unique maximal two-sided
$2$-ideal which does not contain any non-zero identity $2$-morphisms. By \cite[Theorem~15]{MMMZ},
there is a bijection between the sets of equivalence classes of simple transitive $2$-representations
of the $2$-categories $\underline{\widetilde{\cS}}$ and $\widetilde{\cC}$.

At the same time, from Corollary~\ref{corcard3} and the definition of Soergel bimodules 
it follows easily that the $2$-categories $\cC$ and  $\widetilde{\cC}$ are biequivalent.
Therefore there is a bijection between the sets of equivalence classes of simple transitive $2$-representations
of the $2$-categories $\cC$ and $\widetilde{\cC}$. The claim follows.
\end{proof}

The best way to explicitly define a bijection given by Theorem~\ref{thm5} is to use 
the approach to simple transitive $2$-representations via (co)algebra objects as
developed in \cite{MMMT,MMMZ}.

\section{Quiver and relations for the underlying category}\label{s4}

\subsection{Zig-zag categories associated to graphs}\label{s4.1}

Let $\Omega$ be an unoriented graph without loops. 
Let $\widetilde{\Omega}$ be the oriented graph obtained from $\Omega$ by
replacing each unoriented edge of $\Omega$ by a pair of oppositely oriented edges 
between the same vertices. Here is an example:
\begin{displaymath}
\Omega=\xymatrix{1\ar@{-}[r]&2\ar@/^3pt/@{-}[r]\ar@/_3pt/@{-}[r]&3},\qquad
\widetilde{\Omega}=
\xymatrix{1\ar@/^3pt/[r]&2\ar@/^3pt/[l]\ar@/^7pt/[r]\ar@/_3pt/[r]&3\ar@/^7pt/[l]\ar@/_3pt/[l]}.
\end{displaymath}
We denote by $\mathcal{A}^{\Omega}$ the quotient of the path category of $\widetilde{\Omega}$
by the following relations:
\begin{itemize}
\item any path of length three is zero;
\item any path of length two between different vertices is zero;
\item all paths of length two that start and end at the same vertex are equal.
\end{itemize}
The category $\mathcal{A}^{\Omega}$ corresponds to the classical {\em zig-zag} algebra associated to 
$\widetilde{\Omega}$, cf. \cite{HK,Du,ET}. The category $\mathcal{A}^{\Omega}$ is graded by path lengths.
Note that the path algebra of $\mathcal{A}^{\Omega}$ is not unital if $\Omega$ has infinitely many vertices.

\subsection{Cell $2$-representations}\label{s4.2}

In this subsection we assume that $(W,S)$ is of any type. For a fixed $s\in S$,
consider an unoriented graph $\Lambda^{(s)}$ whose set of vertices is 
$\Lambda^{(s)}_0:=\mathcal{L}_s$ and whose set of unoriented arrows is 
\begin{displaymath}
\Lambda^{(s)}_1:= \{\{u,v\}\in \mathcal{L}_s\times\mathcal{L}_s\,:\,
u=tv>v, \text{ for some }t\in S\}.
\end{displaymath}
Note that the graph $\Lambda^{(s)}$ is an unlabeled tree which might be infinite.
We denote by $\mathcal{A}^{(s)}$ the category $\mathcal{A}^{\Lambda^{(s)}}$.
For example, if $\Gamma$ is the graph
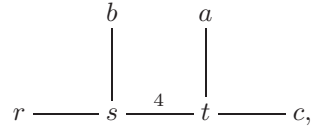
\begin{displaymath}
\xymatrix{
&b\ar@{-}[d]&a\ar@{-}[d]&\\
r\ar@{-}[r]&s\ar@{-}[r]^4&t\ar@{-}[r]&c,
}
\end{displaymath}
then the associated graph $\Lambda^{(s)}$ is as follows:
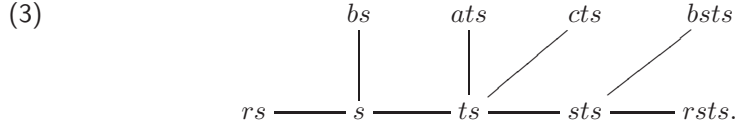
\begin{equation}\label{edq7}
\xymatrix{
&bs\ar@{-}[d]&ats\ar@{-}[d]&cts\ar@{-}[dl]&bsts\ar@{-}[dl]\\
rs\ar@{-}[r]&s\ar@{-}[r]&ts\ar@{-}[r]&sts\ar@{-}[r]&rsts.
}
\end{equation}

\begin{proposition}\label{prop21}
The category $\mathcal{A}^{(s)}$ is isomorphic to the underlying category of the cell
$2$-representation $\mathbf{C}_{\mathcal{L}_s}$ of $\cS$ (and hence also of $\underline{\cS}$).
\end{proposition}

\begin{proof}
Let $\mathcal{B}^{(s)}$ denote the underlying category of $\mathbf{C}_{\mathcal{L}_s}$.
For $w\in\mathcal{L}_s$, we denote by $P_w$ the indecomposable $\mathcal{B}^{(s)}$-projective
corresponding to $w$. We also denote by $L_w$ the simple top of $P_w$.

Recall (see e.g. \cite[Corollary~5.2.4]{Ir}) that, for $t\in S$ and  $w\in W$, we have
\begin{displaymath}
\underline{H}_t\underline{H}_w=
\begin{cases}
v\underline{H}_w+v^{-1}\underline{H}_w, & tw<w;\\
\displaystyle
\underline{H}_{tw}+\sum_{x<w,tx<x}\mu(x,w)\underline{H}_{x}, & tw>w;
\end{cases}
\end{displaymath}
where $\mu(x,w)$ is the {\em Kazhdan-Lusztig $\mu$-function}, see \cite[Subsection~2.2]{KL}.
By construction of $\mathbf{C}_{\mathcal{L}_s}$, this implies that, for $t\in S$ and $x,w\in\mathcal{L}_s$,
the multiplicity $m_{x,w}^{(t)}$ of $P_x$ as a direct summand of  $B_t\cdot P_w$ is given by
\begin{equation}\label{eq3}
m_{x,w}^{(t)}=
\begin{cases}
2, & x=w\text{ and } tw<w;\\
1, & x=tw>w;\\
\mu(x,w), & tw>w, x<w,tx<x;\\
0, & \text{ otherwise}.
\end{cases}
\end{equation}
In the graded picture, we additionally
have that, for $x=w$ and $tw<w$, the two summands $P_w$ appearing in $B_t\cdot P_w$ are, in fact,
$P_w\langle -1\rangle$ and $P_w\langle 1\rangle$ and all other appearing summands have no grading shift.
By adjunction, see e.g. \cite[Lemma~8]{AM}, $m_{x,w}^{(t)}$ coincides with the 
the composition multiplicity of $L_w$ in $B_t\cdot L_x$. 

Now, if $x\in \mathcal{L}_s$, then there is a unique $t\in S$ such that $tx<x$. 
The computation above implies that $B_t\cdot L_x$ has two simple subquotients isomorphic
to $L_x\langle -1\rangle$ and $L_x\langle 1\rangle$ while all other summands $L_w$ are in degree zero
and are killed by $B_t$. We have for such $w$
\begin{displaymath}
\mathrm{Hom}(B_t\cdot L_x,L_w) \cong \mathrm{Hom}(L_x,B_t\cdot L_w)=0,
\end{displaymath}
thus $L_x\langle 1\rangle$ is the simple top of $B_t\cdot L_x$.
Similarly, $L_x\langle -1\rangle$ is the simple socle of $B_t\cdot L_x$.
Note that the module $B_t\cdot L_x$ is both projective and injective by \cite[Theorem~2]{KMMZ}.
As we have just established that $B_t\cdot L_x$ has simple top, we obtain $B_t\cdot L_x\cong P_x\langle 1\rangle$.

As $\mathcal{P}$ is positively graded, the degree zero part of $B_t\cdot L_x$ is semi-simple. 
Hence the number of arrows from $x$ to $w$ in $\mathcal{B}^{(s)}$
coincides with the number of arrows from $w$ to $x$ in $\mathcal{B}^{(s)}$ and this also 
coincides with the multiplicity of $L_w$ in $B_t\cdot L_x$.
If $w=rx$, for some $r\in S$, then the multiplicity of $L_w$ in $B_t\cdot L_x$  is equal to $1$ 
by formula~\eqref{eq3}.

If the multiplicity  of $L_w$ in $B_t\cdot L_x$ is non-zero and $w\neq rx$, for any $r\in S$, then  $x<w$ by formula~\eqref{eq3}. 
In that case, the multiplicity of
$L_x$ in $B_{t'}\cdot L_w$, where $t'\in S$ is the unique element such that $B_{t'}\cdot L_w\neq 0$,
is also non-zero. Then $w<x$ by formula~\eqref{eq3} implying $w=x$, a contradiction.

To sum up, we established that the only $L_w$ which appear in $B_t\cdot L_x$ are those for which 
$w=rx$, for some $r\in S$, and these simples appear with multiplicity $1$. This implies that the
Cartan matrices of $\mathcal{B}^{(s)}$ and $\mathcal{A}^{(s)}$ coincide and now it is easy to 
construct inductively an explicit isomorphism between $\mathcal{B}^{(s)}$ and $\mathcal{A}^{(s)}$ by rescaling, 
if necessary, all arrows, starting the induction from the initial vertex $s$. The claim follows.
\end{proof}

\subsection{Rooted graphs and their pointed union}\label{ssrooted}

Here we introduce the notion of the one point union of rooted graphs, which we use in Subsection \ref{s4.3}.
By a {\em rooted graph}, we mean a pair $(\Xi,a)$ of a graph $\Xi$ and a {\em root} $a\in \Xi_0$.
Let $(\Xi,a)$ and $(\Xi',a')$ be two rooted graphs. The {\em one point union} $(\Xi,a)\vee(\Xi',a')$ of 
$(\Xi,a)$ and $(\Xi',a')$ is the graph obtained from the disjoint union 
$(\Xi,a)\coprod(\Xi',a')$ of $(\Xi,a)$ and $(\Xi',a')$ by 
identifying the roots $a$ and $a'$. The graph $(\Xi,a)\vee(\Xi',a')$ is naturally rooted
with the root being the identified vertex $a=a'$.
Here is an example, where $\bullet$ denotes a root and
$\circ$ an ordinary vertex:
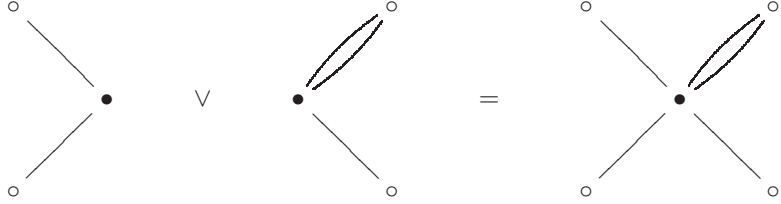
\begin{displaymath}
\xymatrix{\circ\ar@{-}[rd]&&&&\circ\ar@/^3pt/@{-}[ld]\ar@/_3pt/@{-}[ld]&&
\circ\ar@{-}[rd]&&\circ\ar@/^3pt/@{-}[ld]\ar@/_3pt/@{-}[ld]\\
&\bullet&\vee&\bullet&&=&&\bullet&\\
\circ\ar@{-}[ru]&&&&\circ\ar@{-}[lu]&&\circ\ar@{-}[ru]&&\circ\ar@{-}[lu]} 
\end{displaymath}

\subsection{Simple transitive $2$-representations for finitary small quotients}\label{s4.3}

We assume that $(W,S)$ is indecomposable and that 
$\underline{\cS}$ is finitary. The purpose of the subsection is to describe the underlying category of each gradable simple transitive $2$-representation that is not covered in Subsection \ref{s4.2}. They still corresponds to the zig-zag algebras of  certain graphs. We explicitly determine this graph in terms of the Coxeter-Dynkin diagram $\Gamma$ of $(W,S)$.

If $\Gamma$ is an unlabeled tree, then 
any simple transitive $2$-representation of $\underline{\cS}$ is a cell $2$-representation
by Proposition~\ref{prop3} and its underlying category is described by 
Proposition~\ref{prop21}. 
Because of this, in what follows we assume that 
$\Gamma$ has a full subgraph of the form \eqref{eq2}.
 Then we have the
$2$-category $\underline{\widetilde{\cS}}$ as in Subsection~\ref{s3.5}.

Let $\mathbf{M}$ be a gradable simple transitive $2$-representation of $\underline{\cS}$
and $\mathbf{N}$ the corresponding simple transitive $2$-representation of $\underline{\widetilde{\cS}}$
given by Theorem~\ref{thm5}. From the proofs of Theorem~\ref{thm5} and
\cite[Theorem~15]{MMMZ} it follows that the underlying category of
$\mathbf{N}$ is isomorphic to a full subcategory of the underlying category of 
$\mathbf{M}$. As $\mathbf{N}$ is gradable, the underlying category of $\mathbf{N}$ is explicitly
described in \cite{KMMZ,MT} and is known to be of the form $\mathcal{A}^{\Omega}$,
where $\Omega$ is a simply laced Dynkin diagram. There is also an additional datum on $\Omega$,
given by considering $\Omega$ as a bipartite graph $\Omega_0=\Omega_0^{(s)}\coprod \Omega_0^{(t)}$
where, for $r\in\{s,t\}$, we have  $u\in \Omega_0^{(r)}$ if and only if
$B_r\cdot L_u\neq 0$.

By deleting the labeled edge $\{s,t\}$, the connected graph $\Gamma$ splits into a disjoint union
of two connected graphs: the graph $\Gamma^{(s)}$ which is the connected component containing $s$,
and the graph $\Gamma^{(t)}$ which is the connected component containing $t$. 
For $r\in\{s,t\}$, we denote by $S^{(r)}$ the set of vertices in $\Gamma^{(r)}$,
by $W^{(r)}$ the parabolic subgroup of $W$ generated by $S^{(r)}$, and by
$\underline{\cS^{(r)}}$ the small quotient $2$-category of Soergel bimodules associated with $(W^{(r)},S^{(r)})$.
Let $\mathcal{L}^{(r)}_r$ be the left cell in $W^{(r)}$
containing $r$. 

Note that $\Gamma^{(r)}$ is an unlabeled tree and hence any simple transitive 
$2$-representation of $\underline{\cS}^{(r)}$ is a cell $2$-representation
by Proposition~\ref{prop3} and its underlying category is described by 
Proposition~\ref{prop21}. The cell-2 representation for the left cell $\mathcal{L}^{(r)}_r$ is of the form $\mathcal{A}^{\Lambda^{(r)}}$, where $\Lambda^{(r)}$ is as in Subsection \ref{s4.2} with $(W,S)$ replaced by $(W^{(r)},S^{(r)})$ (which is in this case isomorphic to $\Gamma^{(r)}$).

We now construct a graph $\Theta$ via a sequence of pointed unions (see Subsection \ref{ssrooted}).
Consider $\Lambda^{(r)}$ as a rooted graph with root $r$. Letting $\Omega_0^{(s)}=\{u_1,u_2,\dots,u_p\}$
and $\Omega_0^{(t)}=\{w_1,w_2,\dots,w_q\}$,
\begin{itemize}
\item set $\Theta(0):=\Omega$ and consider it as a rooted graph with root $u_1$;
\item for $i$ from $1$ to $p$, define $\Theta(i)$ as the one point union of $\Lambda^{(s)}$
and $\Theta(i-1)$ and then change the root of $\Theta(i)$ to $u_{i+1}$, if $i<p$,
or to $w_1$, if $i=p$;
\item for $j$ from $1$ to $q$, define $\Theta(p+j)$ as the one point union of $\Lambda^{(t)}$
and $\Theta(p+j-1)$ and then change the root of $\Theta(p+j)$ to $w_{j+1}$, if $j<q$,
or forget the root if $j=q$;
\item define the graph $\Theta$ as the unrooted graph $\Theta(p+q)$.
\end{itemize}

Informally, the graph $\Theta$ is obtained from $\Omega$ by attaching at each vertex in 
$\Omega_0^{(s)}$ a copy of $\Lambda^{(s)}$ at $s\in \Lambda^{(s)}$ and attaching at each vertex in $\Omega_0^{(t)}$ a copy of $\Lambda^{(t)}$ at $t$. 

As an example, below we rearrange the graph~\eqref{edq7} such that  the solid part 
depicts $\Omega$, the dashed parts show the attached copies of $\Lambda^{(s)}$
and the dotted part shows the attached copy of $\Lambda^{(t)}$.
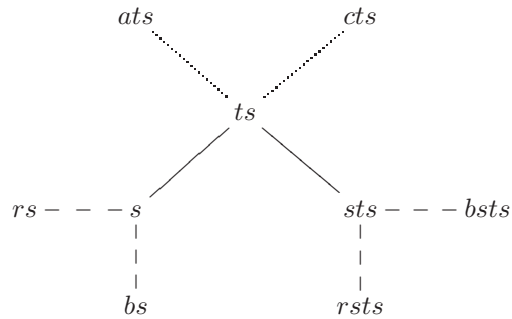
\begin{displaymath}
\xymatrix{
&ats&&cts&\\
&&ts\ar@{.}[lu]\ar@{.}[ru]&&\\
rs&s\ar@{-}[ru]\ar@{--}[l]\ar@{--}[d]&&sts\ar@{-}[lu]\ar@{--}[r]\ar@{--}[d]&bsts\\
&bs&&rsts&
}
\end{displaymath}


\begin{theorem}\label{thmmain2}
The category $\mathcal{A}^{\Theta}$ is isomorphic to the underlying category of the
$2$-rep\-re\-sen\-ta\-ti\-on $\mathbf{M}$ of $\underline{\cS}$. 
\end{theorem}

\begin{proof}
Let $\mathcal{B}$ be the the underlying category of the
$2$-rep\-re\-sen\-ta\-ti\-on $\mathbf{M}$ of $\underline{\cS}$.
Since $\mathcal{B}$ is graded, the same argument as the one used in the proof of 
Proposition~\ref{prop21} implies that $\mathcal{B}$ is isomorphic to 
$\mathcal{A}^{\Theta'}$, for some graph $\Theta'$. So, we just need to check
that this $\Theta'$ is isomorphic to $\Theta$ constructed above.

The subgraph of $\Theta'$ which is isomorphic to $\Omega$ is uniquely determined 
by the fact that $\mathbf{N}$ is isomorphic to a full subcategory of the underlying 
category of $\mathbf{M}$. See above. 

Take now some vertex $u\in\Omega^{(s)}_0$, viewed as a vertex of $\Theta'$ and consider
the additive closure $\mathcal{C}$ of all $B_w\cdot L_u$, where $w\in\mathcal{L}_s^{(s)}$.
The action of $\underline{\cS^{(s)}}$ preserves $\mathcal{C}$ by construction and it is
easy to see that the corresponding $2$-representation of $\underline{\cS^{(s)}}$
is the cell $2$-representation corresponding to $\mathcal{L}_s^{(s)}$. 
Note that the underlying category of this $2$-representation is isomorphic to 
$\mathcal{A}^{\Lambda^{(s)}}$. This argument
works for any $u\in\Omega^{(s)}_0$ and a similar argument 
(with $\mathcal{A}^{\Lambda^{(s)}}$ replaced by $\mathcal{A}^{\Lambda^{(t)}}$) 
works for any  $u\in\Omega^{(t)}_0$. Consequently, there is a natural embedding of 
$\Theta$ into $\Theta'$, and we are left to show that $\Theta'$ has no extra edges.

Let $u_1,u_2\in \Omega^{(s)}_0$ be two different vertices and $\mathcal{C}_1$ and
$\mathcal{C}_2$ the corresponding $2$-rep\-re\-sen\-ta\-ti\-ons of $\underline{\cS^{(s)}}$
constructed in the previous paragraph. Existence of an edge in $\Theta'$ connecting
a vertex in $\mathcal{C}_1$ and a vertex in $\mathcal{C}_2$ implies existence of a
non-trivial discrete self-extension for the cell $2$-representation of 
$\underline{\cS^{(s)}}$ in the sense of \cite[Subsection~5.2]{CM}.
This is, however, prohibited by \cite[Theorem~25]{CM}.

Let $u\in \Omega^{(s)}_0$ and $v\in \Omega^{(t)}_0$, and let $p$ be a vertex in the copy of $\Gamma^{(s)}$ in $\Theta$  attached to $u$ and $q$ be a vertex in the copy of $\Gamma^{(t)}$ attached to $v$.
Suppose there is an edge in $\Theta'$ between $p$ and $q$.
Letting $p'$ be the element in $\mathcal L _s$ corresponding to $p$ and $q'$ be the elements in $\mathcal L_t$ corresponding to $q$, we have \[\Hom_\mathcal{B}(B_{(q')^{-1}}B_{p'}L_u,L_v)\cong\Hom_\mathcal{B}(B_{p'}L_u,B_{q'}L_v)\supseteq \Hom_\mathcal{B}(P_p,P_q)\neq 0,\]
since $B_{p'}L_u$ contains $P_p$, the projective at $p$ in $\mathcal B$ and $B_{q'}L_v$ contains $P_q$.
 In particular, $B_{(q')^{-1}}B_{p'}L_u\neq 0$ and $(q')^{-1}p'\in\mathcal{J}$. Writing  $(q')^{-1}p'=txys$, where $x$ is in the parabolic subgroup generated by $S^{(t)}\setminus t$ and $y$ is in the parabolic subgroup generated by $S^{(s)}\setminus s$, we see that $(q')^{-1}p'\in\mathcal{J}$ implies $x=y=e$.
Therefore,
the edge between $p$ and $q$ comes from an edge in $\Omega$. This completes the proof.
\end{proof}

\vspace{2mm}

\noindent
H.~K.: Max Planck Institute for Mathematics, Vivatsgasse 7, 
D-53111, Bonn,\\ GERMANY, email:  {\tt hankyung\symbol{64}mpim-bonn.mpg.de}

\noindent
V.~M.: Department of Mathematics, Uppsala University, Box. 480,
SE-75106, Uppsala,\\ SWEDEN, email: {\tt mazor\symbol{64}math.uu.se}

\end{document}